\documentclass[11pt,letterpaper,reqno]{amsart}

\usepackage{amssymb,amsmath,amsthm,amsfonts}

\newtheorem{theorem}{Theorem}[section]
\newtheorem{lemma}[theorem]{Lemma}

\numberwithin{equation}{section}

\usepackage{bookmark}
\usepackage{hyperref}
\hypersetup{pdfstartview={FitH}}

\begin{document}

\title{Convergence of lacunary SU(1,1)-valued trigonometric products}

\author[J. Rup\v{c}i\'{c}]{Jelena Rup\v{c}i\'{c}}
\address{Jelena Rup\v{c}i\'{c}, Faculty of Transport and Traffic Sciences, University of Zagreb, Vukeli\'{c}eva 4, 10000 Zagreb, Croatia}
\email{jelena.rupcic@fpz.hr}

\date{\today}


\begin{abstract}
This note attempts to study lacunary trigonometric products with values in the matrix group $\textup{SU}(1,1)$ in analogy with lacunary trigonometric series. The central questions are the characterization of their convergence in an appropriately defined $\textup{L}^p$-metric and the characterization of their convergence almost everywhere. These can be interpreted as nonlinear analogues of the classical results by Zygmund and Kolmogorov.
\end{abstract}

\maketitle


\section{Introduction}

Throughout the text $(A_n)_{n\in\mathbb{Z}}$ will always be a sequence of positive numbers and $(B_n)_{n\in\mathbb{Z}}$ will be a sequence of complex numbers satisfying $A_n^2-|B_n|^2=1$ for every $n\in\mathbb{Z}$. For each $t$ from the one-dimensional torus $\mathbb{T}\equiv\mathbb{R}/\mathbb{Z}$ we can consider the infinite product of matrices
\begin{equation}\label{eq:nftdef}
\prod_{n=-\infty}^{\infty} \begin{bmatrix} A_n & B_n e^{2\pi i n t} \\ \overline{B_n} e^{-2\pi i n t} & A_n \end{bmatrix},
\end{equation}
which we call the \emph{$\textup{SU}(1,1)$ trigonometric product} with the coefficients $A_n$, $B_n$ or the \emph{discrete-time $\textup{SU}(1,1)$ nonlinear Fourier transform} of the sequence of pairs $(A_n,B_n)$. The name comes from the fact that the matrices appearing in \eqref{eq:nftdef} lie in the group
\[ \textup{SU}(1,1) := \left\{ \begin{bmatrix} A & B \\ \overline{B} & \overline{A} \end{bmatrix} \,:\, A,B\in\mathbb{C},\ |A|^2-|B|^2=1 \right\}. \]
Convergence of the infinite product \eqref{eq:nftdef} is interpreted in the sense $\lim_{N\to\infty}\prod_{n=-N}^{N}$. Similarly as with trigonometric series, one can study various modes of converges and impose sufficient and/or necessary conditions on the coefficients $A_n$, $B_n$ for each of these modes. It is important to perform the multiplications in the right order, since in general the above matrices do not commute.

Infinite products \eqref{eq:nftdef} appear in the study of orthogonal polynomials on the unit circle; see the extensive books by Simon \cite{S05a,S05b}. A self-contained introduction to their theory and deduction of their basic properties can be found in the lecture notes by Tao and Thiele \cite{TT03}. The continuous-time analogue of \eqref{eq:nftdef} is the so-called \emph{Dirac scattering transform} (see \cite{MTT03}), arising from the eigenproblem for the Dirac operator. The latter transform is also a particular case of the AKNS systems introduced in \cite{AKNS74} and \cite{ZS71}.

An alternative normalization preferred in \cite{TT03} is
\begin{equation}\label{eq:coeff}
A_n = \frac{1}{\sqrt{1-|F_n|^2}},\quad B_n = \frac{F_n}{\sqrt{1-|F_n|^2}},
\end{equation}
where $(F_n)_{n\in\mathbb{Z}}$ is now a sequence of complex numbers in the open unit disk. We will say that the infinite product \eqref{eq:nftdef} \emph{has $\ell^p$ coefficients} for some $0<p<\infty$ if either of the two mutually equivalent conditions
\[ \sum_{n\in\mathbb{Z}} |B_n|^p < \infty,\quad \sum_{n\in\mathbb{Z}} |F_n|^p < \infty \]
is satisfied. Moreover, for $p=2$ this is also easily seen to be equivalent with any of the conditions
\[ \sum_{n\in\mathbb{Z}} \log(A_n^2+|B_n|^2) < \infty,\quad \sum_{n\in\mathbb{Z}} \log A_n < \infty,\quad
\prod_{n\in\mathbb{Z}} (A_n^2+|B_n|^2) < \infty,\quad \prod_{n\in\mathbb{Z}} A_n < \infty. \]

The analogy between the trigonometric product \eqref{eq:nftdef} and the trigonometric series
\begin{equation}\label{eq:trigser}
\sum_{n=-\infty}^{\infty} D_n e^{2\pi i n t}
\end{equation}
is more than just formal. For instance, Tao and Thiele \cite{TT03} modified the  continuous-time approach of Christ and Kiselev \cite{CK01a,CK01b} to show that the product \eqref{eq:nftdef} with $\ell^p$ coefficients for $p<2$ converges a.e.\@ on $\mathbb{T}$. This can be viewed as a nonlinear analogue of the classical result for the series \eqref{eq:trigser} attributed to Menshov, Paley, and Zygmund \cite{Z59}. An alternative and more quantitative proof of this fact was given by Oliveira e Silva \cite{O17}. It is known that the trigonometric series \eqref{eq:trigser} with $\ell^2$ coefficients also converges a.e.; this is the celebrated result of Carleson \cite{C66}. However the same is only conjectured for the matrix product \eqref{eq:nftdef} with $\ell^2$ coefficients; see the papers \cite{CK01b,K12,KOR19,MTT03,MTT12} for discussions of that one and related open problems. 



The present note attempts to initiate the study of lacunary versions of \eqref{eq:nftdef}. A strictly increasing sequence of positive integers $(m_j)_{j=1}^{\infty}$ is said to be \emph{$q$-lacunary} for some $q>1$ if $m_{j+1}\geq q m_j$ for each $j$. In that case we can define the \emph{associated $q$-lacunary $\textup{SU}(1,1)$ trigonometric product} as the infinite product of matrices
\begin{equation}\label{eq:lacnftdef}
\prod_{j=1}^{\infty} \begin{bmatrix} A_j & B_j e^{2\pi i m_j t} \\ \overline{B_j} e^{-2\pi i m_j t} & A_j \end{bmatrix},
\end{equation}
where the coefficients $A_j>0$, $B_j\in\mathbb{C}$ still satisfy the relation $A_j^2-|B_j|^2=1$ for every $j$. We can again represent them as in \eqref{eq:coeff}, so that indeed $F_j=B_j/A_j$. This setting is related to the work of Golinskii \cite{G04} and Simon \cite{S05b} on sparse Verblunsky coefficients, but here we are concerned with different questions. A lot of work was done on lacunary versions of the trigonometric series,
\begin{equation}\label{eq:lactrigser}
\sum_{j=1}^{\infty} D_j e^{2\pi i m_j t},
\end{equation}
throughout the first half of the twentieth century (see Zygmund's book \cite[\S{}V.6]{Z59}) and we would like to develop a parallel theory in the nonlinear setting of the product \eqref{eq:lacnftdef}. Already in the 1930s, Zygmund showed that the lacunary series \eqref{eq:lactrigser} converges in the $\textup{L}^p$-quasinorm for some fixed $0<p<\infty$ if and only if its coefficients form an $\ell^2$ sequence; this is also a consequence of Zygmund's inequality \cite[Thm.\,V.8.20]{Z59}. Moreover, Kolmogorov \cite{K24} proved that every lacunary series \eqref{eq:lactrigser} with $\ell^2$ coefficients converges a.e.\@ (also see \cite[Thm.\,V.6.3]{Z59}). Conversely, Zygmund \cite{Z30} showed that convergence of \eqref{eq:lactrigser} on a set of positive measure implies that the sequence of its coefficients has to belong to $\ell^2$ (also see \cite[Thm.\,V.6.4]{Z59}). These classical results serve as motivation for the present paper.

In order to study convergence of $\textup{SU}(1,1)$ products we first have to choose a metric function on that group. Let us define a function $\rho\colon \textup{SU}(1,1)\times \textup{SU}(1,1)\rightarrow\mathbb R$ by
$$\rho(G_1,G_2) := \log\left( 1+\left\| G_1^{-1}G_2-I_2\right\| _{op}\right) ,$$
where $I_2$ is the 2-dimensional unit matrix and $\|\cdot\|_{op}$ denotes the operator norm (i.e.\@ the spectral norm) of a matrix. It is straightforward to verify that function $\rho$ is a complete metric on $\textup{SU}(1,1)$. The same metric $\rho$ was considered by Oliveira e Silva \cite{O17}. The metric $\rho$ is invariant with respect to the left multiplication, i.e. \begin{equation*}
\rho\left(GG_1, GG_2\right)=\rho\left(G_1, G_2\right), \textnormal{ for } G,G_1,G_2\in\textup{SU}(1,1).
\end{equation*}
Furthermore, let us denote the set of all measurable functions $g\colon\mathbb T\rightarrow \textup{SU}(1,1)$ by $\textnormal M(\mathbb T,\textup{SU}(1,1))$. Hence, $$\textnormal M(\mathbb T,\textup{SU}(1,1)):=\left\lbrace g=\begin{bmatrix}
a & b \\
\overline{b} & \overline{a} 
\end{bmatrix}:a,b\colon\mathbb T\to\mathbb C\textnormal{ measurable},\;|a(t)|^2-|b(t)|^2=1, \textnormal{for each } t\in\mathbb T\right\rbrace $$ and we identify functions that are equal a.e. We can now define $d_p\colon \textnormal M(\mathbb T,\textup{SU}(1,1))\times \textnormal M(\mathbb T,\textup{SU}(1,1))\to\left[0,\infty \right] $, for $p> 0$ as follows. For $p\ge 1$ let $d_p$ be given by $$d_p(g_1,g_2) := \left\|\rho\left( g_1(t),g_2(t)\right)  \right\|_{\textup{L}^p_t(\mathbb T)}.$$
On the other hand, for $0<p<1$ let $d_p$ be given by $$d_p(g_1,g_2) := \left\|\rho\left( g_1(t),g_2(t)\right)  \right\|_{\textup{L}^p_t(\mathbb T)}^p.$$ 
We define $I\colon\mathbb T\to\textup{SU}(1,1)$ with $I(t)=I_2$, for every $t\in \mathbb T$ and for  $p>0$ we denote $$\textup L^p\left( \mathbb T,\textup{SU}(1,1)\right) := \left\lbrace g\in \textnormal M(\mathbb T,\textup{SU}(1,1)):d_p(I,g)<\infty \right\rbrace.$$ Every partial product of \eqref{eq:lacnftdef} now lies in this set. Also, for $g=\begin{bmatrix}
a & b \\
\overline{b} & \overline{a} 
\end{bmatrix}\in\textnormal M(\mathbb T,\textup{SU}(1,1))$ we can write, more explicitly, \begin{equation}\label{eq:formuladp}
d_p(I,g)=\left\|\log\left(1+|a(t)-1|+|b(t)|\right)\right\| _{\textup L_t^p(\mathbb T)} .\end{equation}
It is an easy exercise to verify that $\left( \textup L^p\left( \mathbb T,\textup{SU}(1,1)\right), d_p\right) $ is a complete metric space, for every $p>0$. 

Now we are ready to state our main result about convergence with respect to the above metric of a $q$-lacunary $\textup{SU}(1,1)$ trigonometric product. Recall that this product is said to have $\ell^2$ coefficients if \begin{equation}\label{eq:proofcondition} 
	\sum_{j=1}^{\infty} \log(A_j^2+|B_j|^2) < \infty.\end{equation}
\begin{theorem}\label{Tmconvmetric}
	Let $(m_j)_{j=1}^{\infty}$ be a $q$-lacunary sequence with $q\ge 2$ and take an arbitrary $p>0$. The infinite product \eqref{eq:lacnftdef} converges in the metric space $\left( \textup L^p\left( \mathbb T,\textup{SU}(1,1)\right), d_p\right) $ if and only if it has $\ell^2$ coefficients.
\end{theorem}
Results about convergence a.e.\ are typically more difficult. We can establish two such results, which can be viewed as partial analogues of the classical results by Kolmogorov \cite{K24} and Zygmund \cite{Z30}.
\begin{theorem}\label{Tmconvae1}
	Let $(m_j)_{j=1}^{\infty}$ be a $q$-lacunary sequence with $q\ge 2$. Then any infinite product \eqref{eq:lacnftdef} with $\ell^2$ coefficients must converge for a.e.\ $t\in\mathbb T$.
\end{theorem}

\begin{theorem}\label{Tmconvae2}
	Let $(m_j)_{j=1}^{\infty}$ be a $q$-lacunary sequence with $q\ge 3$ and suppose that the infinite product \eqref{eq:lacnftdef} converges on a set of positive measure. Then it has $\ell^2$ coefficients.
\end{theorem}

Obvious deficiencies of Theorems~\ref{Tmconvmetric}, \ref{Tmconvae1}, and \ref{Tmconvae2} are the restrictions to the cases $q\geq 2$ and $q\geq 3$. It is quite likely that all three theorems remain valid for each $q>1$, but showing this would require different and more involved proofs. Note that this issue does not appear in the linear theory. The lacunary trigonometric series \eqref{eq:lactrigser} can be sufficiently ``sparsified'' by splitting into several subseries, reducing the proofs of Theorems \ref{Tmconvmetric} and \ref{Tmconvae1} to the case, say, $q\geq 10$. Indeed, many proofs of these classical results proceed that way. On the other hand, matrices in the product \eqref{eq:lacnftdef} do not commute, so the same trick does not apply here. Anyway, we have set a modest first goal to formulate and establish \emph{some} convergence results for lacunary $\textup{SU}(1,1)$ products, leaving the most general cases as interesting open problems. We believe that the above theorems can be considered as a good start for a further investigation.

Before the actual proofs we begin by making several initial observations at the beginning of Section~\ref{proofofconvinmetric}. The proof of Theorem \ref{Tmconvmetric} spans over Section~\ref{proofofconvinmetric} and the proofs of Theorems \ref{Tmconvae1} and \ref{Tmconvae2} are in Section~\ref{proofofconvae} of this paper. The proofs are self-contained and have a combinatorial flavor, as we will be counting certain representations of positive integers. We will give complete proofs of all auxiliary results with two exceptions: the nonlinear Parseval identity by Verblunsky \cite{V35} and the famous weak $\textup{L}^2$ Fourier estimate by Carleson \cite{C66}. 

\section{Convergence in the metric $d_p$}\label{proofofconvinmetric}
We start by studying finite partial products of the infinite product \eqref{eq:lacnftdef},
\begin{equation}\label{eq:lacnfinite} 
	\begin{bmatrix} a_N(t) & b_N(t) \\ \overline{b_N(t)} & \overline{a_N(t)} \end{bmatrix} :=\prod_{j=1}^{N} \begin{bmatrix} A_j & B_j e^{2\pi i m_j t} \\ \overline{B_j} e^{-2\pi i m_j t} & A_j \end{bmatrix},\end{equation}
for a $q$-lacunary sequence $(m_j)_{j=1}^{\infty}$ with $q\ge 2$. For such $q$ frequencies appearing in the exponentials after performing the multiplication of \eqref{eq:lacnfinite} are mutually different for every $N\in\mathbb{N}$. This is seen inductively; also consult \cite{MTT03}. By an induction we can also easily prove the following, slightly more precise result, which will be needed at the very end of the next section. 
\begin{lemma}\label{lema_diff}
	For a $q$-lacunary sequence $(m_j)_{j=1}^{\infty}$ with $q\ge 2$ the frequencies appearing in the expansion of the $\textup{SU}(1,1)$ trigonometric product
	\begin{equation*}
		\prod_{j=M+1}^{N}\begin{bmatrix}
			A_j & B_je^{2\pi im_jt}  \\
			\overline{B_j}e^{-2\pi im_jt} & A_j
		\end{bmatrix},\enspace M,N\in\mathbb{N}_0, \ M<N,\end{equation*}
	are mutually separated by at least $m_{M+1}$.
\end{lemma}
The following recursive formulas apply: \begin{align*}
	a_N(t)&=a_{N-1}(t) A_N+b_{N-1}(t)\overline{B_N}e^{-2\pi im_Nt},\\
	b_N(t)&=a_{N-1}(t) B_Ne^{2\pi im_Nt}+b_{N-1}(t)A_N.
\end{align*}
\begin{lemma}\label{lema1}
	For a $q$-lacunary sequence $(m_j)_{j=1}^{\infty}$ with $q\ge 2$, the following two relations hold:
	\begin{enumerate}
		\item[(a)] $\displaystyle\int_{\mathbb T}\left(\left| a_{N}(t)\right|^2+\left|b_{N}(t) \right|^2\right)\textnormal dt=\prod\limits_{j=1}^{N}\left(A_j^2+|B_j|^2 \right)$
		\item[(b)] $\displaystyle\int_{\mathbb T}\left(\left| a_{N}(t)-A_{1}\cdots A_N\right|^2+\left|b_{N}(t) \right|^2\right)\textnormal dt =\prod\limits_{j=1}^{N}\left(A_j^2+|B_j|^2 \right)-\prod\limits_{j=1}^{N}A_j^2$.
	\end{enumerate}
\end{lemma}
\begin{proof}
	(a) The first statement is proven by induction. It is easy to show that the induction basis holds. In the inductive step we first apply recursive formulas to get \begin{align*}
		\left\| a_{N+1}(t)\right\|_{\textup L_t^2(\mathbb T)}^2&=\left\| a_{N}(t) A_{N+1}+b_{N}(t)\overline{B_{N+1}}e^{-2\pi im_{N+1}t}\right\|_{\textup L_t^2(\mathbb T)}^2,\\[0.15cm] 
		\left\|b_{N+1}(t)\right\|_{\textup L_t^2(\mathbb T)}^2&=\left\|a_{N}(t) B_{N+1}e^{2\pi im_{N+1}t}+b_{N}(t)A_{N+1}\right\|_{\textup L_t^2(\mathbb T)}^2.
	\end{align*}
	Since $q\ge 2$ we know that all frequencies appearing in above trigonometric polynomials are mutually different. Using orthogonality and adding  these two equalities we obtain \[
	\left\| a_{N+1}(t)\right\|_{\textup L_t^2(\mathbb T)}^2+\left\|b_{N+1}(t)\right\|_{\textup L_t^2(\mathbb T)}^2=\left( \left\| a_{N}(t)\right\|_{\textup L_t^2(\mathbb T)}^2
	+\left\|b_{N}(t)\right\|_{\textup L_t^2(\mathbb T)}^2\right) \left(A_{N+1}^2+|B_{N+1}|^2\right). \]
	The first factor is equal to $\prod_{j=1}^{N}(A_j^2+ |B_j|^2 )$ by the inductive hypothesis and the induction step is complete.
	
	\noindent (b) First, notice that $\int_{\mathbb T}\left|a_{N}(t)-A_{1}\cdots A_N\right|^2\textnormal dt= \int_{\mathbb T}\left|a_{N}(t)\right|^2\textnormal dt-\left(A_{1}\cdots A_N\right)^2$, since $\int_{\mathbb T} a_{N}(t)\textnormal dt=A_{1}\cdots A_N.$ This implies \begin{align*}\int_{\mathbb T}\left(\left| a_{N}(t)-A_{1}\cdots A_N\right|^2+\left|b_{N}(t) \right|^2\right)\textnormal dt &= \int_{\mathbb T}\left( \left| a_{N}(t)\right|^2+\left| b_N(t)\right|^2\right) \textnormal dt-\left(A_{1}\cdots A_N\right)^2\\
		&= \prod_{j=1}^{N}\left(A_j^2+|B_j|^2 \right)-\prod_{j=1}^{N}A_j^2 
	\end{align*}
	where the second equality follows from part (a) of the lemma.
\end{proof}
We can write \begin{equation*}
	a_N(t)=\sum_{n\in E_N}C_ne^{2\pi int},\quad b_N(t)=\sum_{n\in F_N}D_ne^{2\pi int}
\end{equation*} where $E_N, F_N\subseteq\mathbb Z$ are the sets of all frequencies for which the Fourier coefficients of trigonometric polynomials $a_N$ and $b_N$, respectively, are nonzero. Notice that Lemma~\ref{lema1} then actually says $$\sum_{n\in E_N\cup F_N}\left(|C_n|^2+|D_n|^2\right)= \prod_{j=1}^{N}\left(A_j^2+|B_j|^2 \right),$$ assuming that for an initial $q$-lacunary sequence $(m_j)_{j=1}^{\infty}$ we have $q\ge 2$.
\begin{proof}[Proof of Theorem~\ref{Tmconvmetric}]
Recall that the convergence of \eqref{eq:lacnftdef} in the metric $d_p$ means that \begin{equation}\label{eq:lacncauchy}
\lim\limits_{M,N\rightarrow\infty}d_p\left( \begin{bmatrix}
a_M & b_M \\
\overline{b_M} & \overline{a_M} 
\end{bmatrix},\begin{bmatrix}
a_N & b_N \\
\overline{b_N} & \overline{a_N} 
\end{bmatrix}\right) = 0,
\end{equation} since the space $\textup L^p\left( \mathbb T,\textup{SU}(1,1)\right)$ is complete. First assume that condition \eqref{eq:proofcondition} is satisfied and that $p\geq 1$, and take $M,N\in\mathbb N$, $M<N$. Using formula \eqref{eq:formuladp} and the fact that metric $\rho$ is invariant with respect to the left multiplication we get\begin{align*}
d_p\left( \begin{bmatrix}
a_M & b_M \\
\overline{b_M} & \overline{a_M} 
\end{bmatrix},\begin{bmatrix}
a_N & b_N \\
\overline{b_N} & \overline{a_N} 
\end{bmatrix}\right) &= d_p\left(I,\begin{bmatrix}
a_{M,N} & b_{M,N} \\
\overline{b_{M,N}} & \overline{a_{M,N}} 
\end{bmatrix} \right)\\
&=\left\|\log\left(1+\left|a_{M,N}(t)-1\right|+\left|b_{M,N}(t)\right|\right)\right\| _{\textup L_t^p(\mathbb T)},\end{align*}
where we denoted \begin{equation*}
\begin{bmatrix}
a_{M,N}(t) & b_{M,N}(t) \\ 
\overline{b_{M,N}(t)} & \overline{a_{M,N}(t)}
\end{bmatrix} := \prod_{j=M+1}^{N} \begin{bmatrix}
A_j & B_je^{2\pi im_jt} \\[0.15cm] 
\overline{B_j}e^{-2\pi im_jt} & A_j
\end{bmatrix}.\end{equation*} 
Put \begin{equation} \label{eq:formulaSmn}
S_{M,N}:= \sum\limits_{j=M+1}^{N}\log\left(A_j^2+|B_j|^2 \right).
\end{equation} 
Since $q\ge 2$, we can apply Lemma \ref{lema1} (b) to the above finite matrix product starting with the $(M+1)$-st term of \eqref{eq:lacnfinite} and get \begin{equation}\label{eq:sumsquares}
\int_{\mathbb T}\left(\left| a_{M,N}(t)-A_{M+1}\cdots A_N\right|^2+\left|b_{M,N}(t) \right|^2\right)\textnormal dt=\prod_{j=M+1}^{N}\left(A_j^2+|B_j|^2 \right)-\prod_{j=M+1}^{N} A_j^2 .
\end{equation}
For $\alpha>0$, we have \begin{align*}E_\alpha& := \left\lbrace t\in\mathbb T:\left(  \log \left(1+\left| a_{M,N}(t)-1\right|+\left|b_{M,N}(t) \right| \right)\right)^p>\alpha\right\rbrace\\
&\;=\left\lbrace  t\in\mathbb T:\left( \left| a_{M,N}(t)-1\right|+\left|b_{M,N}(t) \right|\right)^2>\big( e^{\alpha^{1/p}}-1\big)^2\right\rbrace \\
&\;\subseteq \left\lbrace  t\in\mathbb T:\left| a_{M,N}(t)-1\right|^2+\left|b_{M,N}(t) \right|^2>\frac{1}{2}\big( e^{\alpha^{1/p}}-1\big)^2\right\rbrace .
\end{align*}
where the last inclusion is obtained by applying an elementary inequality $(x+y)^2\le2x^2+2y^2$, for $x,y\in\mathbb R$. Furthermore, notice that \begin{align*} &\int_{\mathbb T}\left( \left| a_{M,N}(t)-1\right|^2+\left|b_{M,N}(t) \right|^2\right)\textnormal dt\\
&\le\int_{\mathbb T}\left(\left(  \left| a_{M,N}(t)-A_{M+1}\cdots A_N\right| +\left|A_{M+1}\cdots A_N-1\right|\right) ^2+\left|b_{M,N}(t) \right|^2\right)\textnormal dt\\
&\le 2\int_{\mathbb T}\left(\left| a_{M,N}(t)-A_{M+1}\cdots A_N\right|^2+\left|b_{M,N}(t) \right|^2\right)\textnormal dt+2\left(A_{M+1}\cdots A_N-1\right)^2,
\intertext{and that is, using \eqref{eq:sumsquares}, equal to }
&=2\Big(\prod\limits_{j=M+1}^{N}\left(A_j^2+|B_j|^2\right) -\prod\limits_{j=M+1}^{N}A_j^2\Big)+2\Big( \prod\limits_{j=M+1}^{N}A_j-1\Big)^2\\
&=2\,\Big( e^{S_{M,N}}-2\prod\limits_{j=M+1}^{N} A_j 
 +1 \Big)\le 2\left(e^{S_{M,N}}-1 \right).
\end{align*}
Chebyshev’s inequality can now be used to see that 
\begin{equation*}
\left|E_\alpha\right|\le \dfrac{2}{\big(e^{\alpha^{1/p}}-1\big)^2 }\int_{\mathbb T}\left( \left| a_{M,N}(t)-1\right|^2+\left|b_{M,N}(t) \right|^2\right)\textnormal dt \le \dfrac{4\left( e^{S_{M,N}}-1\right) }{\big(e^{\alpha^{1/p}}-1\big)^2 }.
\end{equation*}
Also, we know that for $0<p<\infty$ and a complex measurable function $f$ the equality $$\int|f|\textnormal d\mu=\int_{0}^{\infty}\mu\left(\left\lbrace x\in X:|f(x)|>\alpha\right\rbrace\right)\textnormal d\alpha$$ holds. From these two relations we obtain, for $p\geq 1$, \begin{equation}\label{eq:measureofE}
d_p\left( \begin{bmatrix}
	a_M & b_M \\
	\overline{b_M} & \overline{a_M} 
\end{bmatrix},\begin{bmatrix}
	a_N & b_N \\
	\overline{b_N} & \overline{a_N} 
\end{bmatrix}\right)^p =\int\limits_0^{\infty}\left|E_\alpha\right|\,\textnormal d\alpha  
\le \left( e^{S_{M,N}}-1\right)\int\limits_0^{\infty}\dfrac{4\,\textnormal d\alpha}{\big(e^{\alpha^{1/p}}-1\big)^2 }.
\end{equation}
Now for $2<p<\infty$ we conclude that \begin{equation}\label{eq:ocjenazaCp}
C_p := \int\limits_0^{\infty}\dfrac{4\,\textnormal d\alpha}{\big(e^{\alpha^{1/p}}-1\big)^2 } <\infty,
\end{equation}
since the function under the integral is asymptotically equal to $\frac{4}{\alpha^{2/p}}$ when $\alpha\rightarrow 0^+$, and decreases faster then $\frac{1}{\alpha^2}$ when $\alpha\rightarrow\infty$. Letting $M,N\rightarrow\infty$, from assumption \eqref{eq:proofcondition} we get that $$\lim_{M,N\rightarrow\infty}\sum_{j=M+1}^{N}\log(A_j^2+|B_j|^2)=0.$$ That implies $\lim_{M,N\rightarrow\infty}( e^{S_{M,N}}-1) =0$, which, together with  
\eqref{eq:measureofE}, implies \eqref{eq:lacncauchy}. Therefore, the statement is proven for $p>2$. If $1\le p\le 2$, then \eqref{eq:lacncauchy} follows from the previously established case $p>2$ and the monotonicity of the $\textup L^p$ norms on $\mathbb T$. For $0<p<1$ it is known that $d_p\le d_1^p$, again by the monotonicity of the $\textup{L}^p$ quasinorms, so in this case \eqref{eq:lacncauchy} follows from the established convergence for $d_1$.

Conversely, we now assume that the infinite product \eqref{eq:lacnftdef} converges with the respect to the metric $d_p$. Let $M,N\in\mathbb N$, $M<N$ and let us keep the notation from \eqref{eq:formulaSmn}. It is sufficient to consider the case when $0<p<1$, due to the monotonicity of the $\textup L^p$ quasinorms on $\mathbb T$. Let $0<\theta <1$ be chosen such that the equality $\frac{1}{2}=\frac{1-\theta}{p}+\frac{\theta}{4}$ holds. Nonlinear version of Parseval's identity that we will need, says $$ \int_{\mathbb{T}} \log|a_{M,N}(t)| \,\textnormal dt = \sum_{j=M+1}^N \log A_j$$
and the proof of this formula  \cite[\S{}2.1]{TT03} traces back to Verblunsky \cite[pp.~291]{V35}. Furthermore, for $x\ge1$ the next two inequalities hold:  \begin{equation*}
	\log\left(2x^2-1\right)\le4\log x,\quad \sqrt{\log x}\le\log(x+\sqrt{x^2-1}). 
\end{equation*} Using these two inequalities, log-convexity of the $\textup L^p$ quasinorms and the nonlinear version of Parseval's identity we get \begin{align}\label{eq:ocjena sa S_M,N}
S_{M,N}&\le 4\sum\limits_{j=M+1}^{N}\log A_j=4\Big\| \sqrt{\log\left|a_{M,N}(t) \right| }\Big\|_{\textup L^2_t(\mathbb T)}^2\le 4\left\| \log\left( 1+\left| a_{M,N}(t)-1\right|+\left| b_{M,N}(t)\right|\right) \right\|_{\textup L^2_t(\mathbb T)}^2\nonumber\\
&\le 4\left\| \log\left( 1+\left| a_{M,N}(t)-1\right|+\left| b_{M,N}(t)\right|  \right) \right\|_{\textup L^p_t(\mathbb T)}^{2(1-\theta)} \cdot\left\| \log\left( 1+\left| a_{M,N}(t)-1\right|+\left| b_{M,N}(t)\right|  \right) \right\|_{\textup L^4_t(\mathbb T)}^{2\theta}\nonumber \\
&\le 4\left\| \log\left( 1+\left| a_{M,N}(t)-1\right|+\left| b_{M,N}(t)\right|  \right) \right\|_{\textup L^p_t(\mathbb T)}^{2(1-\theta)} \cdot\left( C_4 ( e^{S_{M,N}}-1) \right)^{\theta/2}. 
\end{align}
For the second factor we use \eqref{eq:measureofE} with the same notation as in \eqref{eq:ocjenazaCp} for $p=4$. Now assume the opposite, that $\sum_{j=1}^{\infty}\log(A_j^2+|B_j|^2)=\infty$. Notice that our assumption \eqref{eq:lacncauchy} for $M=j-1$ and $N=j$ implies \begin{equation*}\lim\limits_{j\rightarrow\infty}\left\|\log\left( 1+|A_j -1|+|B_j|\right) \right\|_{\textup L_t^p(\mathbb T)}=0,\end{equation*}
which means that \begin{equation}\label{limopci}
\lim\limits_{j\rightarrow\infty}\log\left(A_j+|B_j|\right) =0.
\end{equation}
Since $\log(A_j^2+|B_j|^2)\le 2\log\left(A_j+|B_j|\right)$, from \eqref{limopci} we have that individual terms of the divergent series $\sum_{j=1}^{\infty}\log(A_j^2+|B_j|^2 )$ tend to $0$. 
Using this we can easily find strictly increasing sequences of indices $(M_k)_{k\in\mathbb N}$ and $(N_k)_{k\in\mathbb N}$ such that $\lim_{k\rightarrow\infty}S_{M_k,N_k}=1$. Once again, from our assumption \eqref{eq:lacncauchy} we get \begin{equation}\label{eq:eq325}
\lim\limits_{k\rightarrow\infty }\left\|\log\left( 1+\left|a_{M_k,N_k}(t) -1\right|+\left|b_{M_k,N_k}(t) \right|\right) \right\|_{\textup L^p_t(\mathbb T)}=0.
\end{equation}
Finally, \eqref{eq:ocjena sa S_M,N} now implies $$S_{M_k,N_k}\le4C_4^{\theta/2}\big( e^{S_{M_k,N_k}}-1\big)^{\theta/2}\left\|\log\left( 1+\left|a_{M_k,N_k}(t) -1\right|+\left|b_{M_k,N_k}(t) \right|\right) \right\|_{\textup L^p_t(\mathbb T)}^{2(1-\theta)}.$$
Letting $k\rightarrow\infty$, from $\lim_{k\rightarrow\infty}S_{M_k,N_k}=1$ and \eqref{eq:eq325} we get a contradiction. Consequently, \eqref{eq:proofcondition} holds. The statement for $p\ge 1$ then follows from the previous case $0<p<1$ and the fact that for $q<1\le p$ we have $d_q^{1/q}\le d_p$. This completes the proof of Theorem \ref{Tmconvmetric}.
\end{proof}

\section{Convergence almost everywhere}\label{proofofconvae}

We start with the proof of Theorem \ref{Tmconvae1} that relies on Carleson's result \cite{C66}.
\begin{proof}[Proof of Theorem~\ref{Tmconvae1}]
Since the metric space $\left(\textup{SU}(1,1),\rho \right) $ is complete, in order to prove convergence of the infinite product \eqref{eq:lacnftdef} at the point $t\in\mathbb T$  it is sufficient to prove
\[\lim\limits_{M,N\rightarrow\infty}\rho\left(\begin{bmatrix}
	a_M(t) & b_M(t) \\
	\overline{b_M(t)} & \overline{a_M(t)} 
	\end{bmatrix},\begin{bmatrix}
		a_N(t) & b_N(t) \\
		\overline{b_N(t)} & \overline{a_N(t)} 
		\end{bmatrix} \right)=\lim\limits_{M,N\rightarrow\infty}\log\left( 1+\left|a_{M,N}(t) -1\right|+\left|b_{M,N}(t) \right|\right)=0,\]
or equivalently \begin{equation}\label{eq:limtoprove}
\lim\limits_{M,N\rightarrow\infty} \left|a_{M,N}(t)-1 \right|=0\enspace\textnormal{ and}\enspace\lim\limits_{M,N\rightarrow\infty} \left|b_{M,N}(t) \right|=0.
\end{equation}
For $M,N\in\mathbb N$, $M<N$ denote $$\widetilde{a}_{M,N}(t) := a_{M,N}(t)-\prod_{j=M+1}^{N}A_j.$$ 
Also, we put \begin{equation*}
\widetilde{a}_{M,N}(t)=\sum_{n\in E_{M,N}}C_ne^{2\pi int},\quad b_{M,N}(t)=\sum_{n\in F_{M,N}}D_ne^{2\pi int}
\end{equation*} where $E_{M,N}, F_{M,N}\subseteq\mathbb Z$ are the sets of all frequencies for which the Fourier coefficients of $\widetilde{a}_{M,N}$ and $b_{M,N}$, respectively, are nonzero. Due to assumption \eqref{eq:proofcondition} there exists $K_k\in\mathbb N$ such that, for all $M,N\ge K_k$, inequality $e^{S_{M,N}}-1\le\frac{1}{k^22^k}$ holds. We can also assume $K_1<K_2<\ldots$. In addition, denote \begin{equation*}
\widetilde{G}_k :=\Bigg\lbrace t\in\mathbb T:\sup_{\substack{N\in\mathbb N \\ N>K_k}}\left|\widetilde{a}_{K_k,N}(t)\right|>\frac{1}{k} \Bigg\rbrace, \quad
G_k:=\Bigg\lbrace t\in\mathbb T:\sup_{\substack{N\in\mathbb N \\ N>K_k}}\left| b_{K_k,N}(t)\right|>\frac{1}{k} \Bigg\rbrace.\end{equation*}
Notice that, since $m_{j+1}\ge 2m_j$, trigonometric polynomials $\widetilde{a}_{K_k,N}$ are all partial sums of the same trigonometric series. By applying linear Carleson's theorem \cite{C66} to that trigonometric series and using the remark after Lemma \ref{lema1} and the fact that $A_j\ge 1$ the next inequality follows with a finite absolute constant $C$:
\begin{align*}\left|\widetilde{G}_k\right| &\le C\left(\frac{1}{k}\right)^{-2}\sum_{n\in E_{K_k,\infty}}|C_n|^2\le C\left(\frac{1}{k}\right)^{-2}\left( \prod_{j=K_k+1}^{\infty}\left(A_j^2+|B_j|^2 \right)-1\right)\\ 
&= C\left(\frac{1}{k}\right)^{-2}\left(e^{S_{K_k,\infty}}-1\right)\le C\left(\frac{1}{k}\right)^{-2}\frac{1}{k^22^k} = \frac{C}{2^k},\end{align*}
where $E_{K_k,\infty}:=\bigcup_{\substack{N\in\mathbb N \\ N>K_k}} E_{K_k,N}$ and  $S_{M,\infty}:= \sum_{j=M+1}^{\infty}\log(A_j^2+|B_j|^2 )$, for $M\in\mathbb N$. Analogously, we get $|G_k|\le \frac{C}{2^k}.$ 

Since $\sum_{k=1}^{\infty}|\widetilde{G}_k|<\infty$, Borel-Cantelli's lemma implies $| \cap_{l=1}^{\infty}\cup_{k=l}^{\infty}\widetilde{G}_k| =0$. Hence, a.e.\ $t\in\mathbb T$ is in only finitely many sets $\widetilde{G}_k$. Therefore, for a.e.\ $t\in\mathbb T$ for all but finitely many $k\in\mathbb N$ for every $N>K_k$ we have \begin{align*}
\left| a_{K_k,N}(t)-1 \right|& \le \left| \widetilde{a}_{K_k,N}(t)\right|+\prod_{j=K_k+1}^{N}A_j-1 \le \frac{1}{k}+\prod_{j=K_k+1}^{N}\left(A_j^2+|B_j|^2\right)-1\\
&\le \frac{1}{k}+e^{S_{K_k,\infty}}-1 \le\frac{2}{k},
\end{align*} 
which implies
$$\sup_{\substack{N\in\mathbb N \\ N>K_k}}\left| a_{K_k,N}(t)-1 \right|\le\frac{2}{k}.$$
Analogously, from $\sum_{k=1}^{\infty}\left|G_k\right|<\infty$ we get that for a.e.\ $t\in\mathbb T$ for all but finitely many $k\in\mathbb N$ for every $N>K_k$ the following inequality holds $$\sup_{\substack{N\in\mathbb N \\ N>K_k}}\left| b_{K_k,N}(t)\right|\le\frac{1}{k}.$$
Let $M,N\in\mathbb N$ and $t\in\mathbb{T}$ be such that $K_k\le M<N$, $t\notin \cap_{l=1}^{\infty}\cup_{k=l}^{\infty}\widetilde{G}_k$ and $t\notin \cap_{l=1}^{\infty}\cup_{k=l}^{\infty}G_k$. Put $g_N(t) := \begin{bmatrix}
a_N(t) & b_N(t) \\
\overline{b_N(t)} & \overline{a_N(t)} 
\end{bmatrix}$. The triangle inequality for $\rho$ implies $$\rho\left( g_M(t),g_N(t)\right) \le\rho\left(g_{K_k}(t),g_M(t)\right)+\rho\left( g_{K_k}(t),g_N(t)\right),$$
and we get \begin{equation*}\;1+\left|a_{M,N}(t)-1\right|+\left|b_{M,N}(t) \right|
\le \left( 1+\left|a_{K_k,M}(t)-1\right|+\left|b_{K_k,M}(t)\right| \right) \left( 1+\left|a_{K_k,N}(t)-1\right|+\left|b_{K_k,N}(t)\right| \right).  \end{equation*}
Now we conclude that \begin{equation*}\left|a_{M,N}(t)-1\right|+\left|b_{M,N}(t) \right|
\leq \left(1+\frac{3}{k}\right)^2 - 1 \leq \frac{15}{k}  \end{equation*}
and furthermore $$\sup_{\substack{M,N\in\mathbb N \\ K_k\le M<N}}\left| a_{M,N}(t)-1\right|\le\frac{15}{k}\quad\textnormal{and}\enspace\sup_{\substack{M,N\in\mathbb N \\ K_k\le M<N}}\left| b_{M,N}(t)\right|\le\frac{15}{k}, $$ which proves the convergence stated in \eqref{eq:limtoprove}, so the proof of Theorem \ref{Tmconvae1} is complete.
\end{proof}

In the proof of Theorem \ref{Tmconvae2} we need the following two lemmas. The first one is the reason why in this theorem we have the condition $q\ge 3$.

\begin{lemma}
	Let $(m_j)_{j=1}^{\infty}$ be a $q$-lacunary sequence with $q\ge 3$. Each $n\in\mathbb Z$ has at most one maximally shortened representation \begin{equation}\label{skracprikaz}
	n=\left( m_{j_1}-m_{j_2}+\cdots+m_{j_{J}}\right)-\left( m_{k_1}-m_{k_2}+\cdots+m_{k_K}\right),
	\end{equation}
	where $J,K\in\mathbb N$ are odd numbers, $j_1<j_2<\cdots<j_J$ and $k_1<k_2<\cdots<k_K$. Here by ``maximally shortened'' we mean a representation in which no further cancellation of terms within the two pairs of parentheses is possible.
\end{lemma}
\begin{proof}
	We will actually prove a stronger result than the one stated in the lemma. The stronger result we are going to prove is the following. Assume that we have \begin{align}
	&\left( m_{j_1}-m_{j_2}+\cdots+(-1)^{J-1}m_{j_J}\right)-\left( m_{k_1}-m_{k_2}+\cdots+(-1)^{K-1}m_{k_K}\right)\nonumber \\
	&=\left( m_{j'_1}-m_{j'_2}+\cdots+(-1)^{J'-1}m_{j'_{J'}}\right)-\left( m_{k'_1}-m_{k'_2}+\cdots+(-1)^{K'-1}m_{k'_{K'}}\right)\label{desna i lijeva strana},
	\end{align}
	where $J,J',K,K'\in\mathbb N_0$, $J+K$ and $J'+K'$ are of the same parity and $j_1<j_2<\cdots<j_J$, $j'_1<j'_2<\cdots<j'_{J'}$, $k_1<k_2<\cdots<k_K$, $k'_1<k'_2<\cdots<k'_{K'}$. If both sides of the equality are maximally shortened, then the left and the right side have the same terms (after possible moving of the terms from one pair of parentheses to the other and after rearrangements). We prove this statement by the induction on $J+J'+K+K'\in\mathbb N_0$. The induction basis $J=J'=K=K'=0$ is trivially satisfied. For the induction step we take one above described equality and without loss of generality we may assume that then $j_J\ge 1$ is the largest number among $j_J,j'_{J'},k_K,k'_{K'}$. Notice that we cannot have  $(-1)^{J-1}m_{j_J}=(-1)^{K-1}m_{k_K}$ because in that case we could cancel these terms on the left hand side. On the other hand, if we had $(-1)^{J-1}m_{j_J}=(-1)^{J'-1}m_{j'_{J'}}$ or $(-1)^{J-1}m_{j_J}=-(-1)^{K'-1}m_{k'_{K'}}$, then the same term could be subtracted from both sides of equality \eqref{desna i lijeva strana} and we could use the inductive hypothesis. Otherwise, we discuss all possible parity combinations of the numbers $J,J',K,K'$. In every of eight possibilities we carefully estimate both the left side and the right side of \eqref{desna i lijeva strana}. In each case, using the assumption that $q\ge 3$, we get a contradiction. 
	
	For instance, if $J,K,J',\text{ and }K'$ are all odd numbers we have $k_K<j_J$ and $j'_{J'}<j_J$. Now it follows   
	\begin{align*}
		&\textnormal{LHS} > m_{j_J}-m_{j_J-1}-m_{k_K}\ge m_{j_J}-\frac{1}{3}m_{j_J}-\frac{1}{3}m_{j_J}=\frac{1}{3}m_{j_J},\\
		&\textnormal{RHS} < m_{j'_{J'}}\le\frac{1}{3}m_{j_J} ,
	\end{align*} where we denoted by LHS the left hand side of \eqref{desna i lijeva strana} and by RHS the right hand side of \eqref{desna i lijeva strana}. This, obviously, leads us to a contradiction.
\end{proof}

\begin{lemma} \label{lema_ocjena_Dn-ova}
	Let $(m_j)_{j=1}^{\infty}$ be a $q$-lacunary sequence with $q\ge 3$ and let $M,N\in\mathbb N$, $M<N$. In addition, assume that $b_{M,N}(t)=\sum_{n\in F}D_ne^{2\pi int}$, where $F\subseteq \mathbb Z$ is the set of all frequencies for which the Fourier coefficients of $b_{M,N}$ are nonzero. Then the following inequality holds:
	$$\sum_{n\in\mathbb Z}\Big|\sum_{\substack{n_1,n_2\in F\\ n_2-n_1=n}} D_{n_1}\overline{D_{n_2}}\Big|^2\le e^{8\sum_{j=M+1}^{N}|B_j|^2}.$$
\end{lemma}
\begin{proof}
	For a fixed $n\in\mathbb Z$ we study all representations of $n$ of the form given in \eqref{skracprikaz}, i.e. \begin{equation}\label{eq:representation}
	n=n_2-n_1=\left( m_{j_1}-m_{j_2}+\cdots+m_{j_J}\right)-\left( m_{k_1}-m_{k_2}+\cdots+m_{k_K}\right),	
	\end{equation}
where $J,K\in\mathbb N$ are odd numbers, $j_1,j_2\ldots,j_J\in\mathbb N$ and $k_1,k_2,\ldots,k_K\in\mathbb N$ are such that $j_1<j_2<\cdots<j_J$ and $k_1<k_2<\cdots<k_K$ holds, but allowing the possibility that some terms cancel. Furthermore, for $l=0,1,2$ denote by $\mathcal{S}_n^l$ the set of all $M+1\le j\le N$ such that the frequency $m_j$ appears in the maximally shortened representation of $n$ exactly $l$ times. Notice that a set $\left\lbrace M+1,\ldots,N\right\rbrace$ is a disjoint union of $\mathcal{S}_n^0$, $\mathcal{S}_n^1$ and $\mathcal{S}_n^2$. Having in mind that every representation \eqref{eq:representation} $n=n_2-n_1$, $n_1,n_2\in F$, can be obtained from the unique maximally shortened representation from the previous lemma by adding terms which correspond to indices from $\mathcal{S}_n^0$, we can conclude
\begin{equation}\label{eq:sumofd-ova}
\sum_{\substack{n_1,n_2\in F\\ n_2-n_1=n}}\big|D_{n_1}D_{n_2}\big|\le\Big(\prod_{j\in\mathcal{S}_n^0}\left(A_j^2+|B_j|^2 \right)\Big) \Big(\prod_{j\in\mathcal{S}_n^1}|A_jB_j|\Big) \Big(\prod_{j\in\mathcal{S}_n^2}|B_j|^2 \Big).
\end{equation}	
Furthermore, notice that for any partition $\left( \mathcal{S}^0,\mathcal{S}^1,\mathcal{S}^2\right)$ of $\left\lbrace M+1,\ldots,N\right\rbrace$ we have at most $2^{|\mathcal{S}_n^1|}$ numbers $n\in\mathbb Z$ such that $$\left( \mathcal{S}_n^0,\mathcal{S}_n^1,\mathcal{S}_n^2\right) =\left( \mathcal{S}^0,\mathcal{S}^1,\mathcal{S}^2\right) .$$
Namely, $j\in\mathcal{S}_n^1$ means that the frequency $m_j$ has to appear either within the first or within the second pair of parentheses in the representation \eqref{eq:representation}, while everything else is uniquely determined. In particular, the signs preceding the frequencies $m_j$ are uniquely determined by the positions of those frequencies within the parentheses; this is seen by reasoning backwards, starting with the rightmost summand. Squaring inequality \eqref{eq:sumofd-ova} and summing over all $n$ that determine the same triple of sets $\mathcal{S}^0_n$, $\mathcal{S}^1_n$, and $\mathcal{S}^2_n$ we get \begin{align*}
&\sum_{\substack{n\in\mathbb Z\\ \left( \mathcal{S}^0_n,\mathcal{S}^1_n,\mathcal{S}^2_n\right) =\left( \mathcal{S}^0,\mathcal{S}^1,\mathcal{S}^2\right) }}\Bigg( \sum_{\substack{n_1,n_2\in F\\ n_2-n_1=n}} \big|D_{n_1}D_{n_2}\big|\Bigg) ^2 \\
&\qquad\qquad \le 2^{|\mathcal S^1|}\Big(\prod_{j\in\mathcal{S}^0}\left(A_j^2+|B_j|^2 \right)^2\Big) \Big(\prod_{j\in\mathcal{S}^1}A_j^2|B_j|^2\Big) \Big(\prod_{j\in\mathcal{S}^2}|B_j|^4 \Big)\\
&\qquad\qquad =\Big(\prod_{j\in\mathcal{S}^0}\left(A_j^2+|B_j|^2 \right)^2\Big) \Big(\prod_{j\in\mathcal{S}^1}2A_j^2|B_j|^2\Big) \Big(\prod_{j\in\mathcal{S}^2}|B_j|^4 \Big).\end{align*}
Finally, taking the sum over all possible choices of partitions $\left(\mathcal{S}^0,\mathcal{S}^1,\mathcal{S}^2\right)$ of $\left\lbrace M+1,\ldots,N\right\rbrace$  and using $1+x\le e^x$, for $x\in\mathbb R$,  the above inequality and the fact that  $A_j^2-|B_j|^2=1$, we get \begin{align*}
\sum_{n\in\mathbb Z}\Big|\sum_{\substack{n_1,n_2\in F\\ n_2-n_1=n}} D_{n_1}\overline{D_{n_2}}\Big|^2&\le\sum_{\left( \mathcal{S}^0,\mathcal{S}^1,\mathcal{S}^2\right)}\sum_{\substack{n\in\mathbb Z\\ \left( \mathcal{S}^0_n,\mathcal{S}^1_n,\mathcal{S}^2_n\right) =\left( \mathcal{S}^0,\mathcal{S}^1,\mathcal{S}^2\right) }}\Bigg( \sum_{\substack{n_1,n_2\in F\\ n_2-n_1=n}} \big|D_{n_1}D_{n_2}\big|\Bigg) ^2\\
&\le\prod_{j=M+1}^{N}\left(\left(A_j^2+|B_j|^2 \right)^2+2A_j^2|B_j|^2+|B_j|^4 \right) \\
&\le \prod_{j=M+1}^{N}\left(A_j^2\left(A_j^2+6|B_j|^2 \right)\right) \\
&= \bigg(\prod_{j=M+1}^{N}\left(1+|B_j|^2 \right)\bigg) \bigg(\prod_{j=M+1}^{N}\left(1+7|B_j|^2\right)\bigg) \\[0.2cm] 
&\le e^{8\sum_{j=M+1}^{N}|B_j|^2}.\qedhere
\end{align*}	
\end{proof}
Now we are ready to prove Theorem \ref{Tmconvae2}. Our proof was inspired by Zygmund's proof of the linear version of this result \cite{Z30}. 
\begin{proof}[Proof of Theorem~\ref{Tmconvae2}]
The proof starts by determining a set $E\subseteq\mathbb T$ of positive measure having property that \begin{equation}\label{limsupb_MN}
\lim_{M,N\to\infty}\sup_{t\in E}\left|b_{M,N}(t) \right|=0. 
\end{equation}
This is a quite standard argument in measure theory, but we give the details for completeness. According to the assumption of the theorem we know that the sequences $(a_N)_{N\in\mathbb N}$ and $(b_N)_{N\in\mathbb N}$ converge pointwise on some set $Z\subseteq\mathbb T$ of positive measure. By Egorov's theorem  these sequences also converge uniformly on some subset $E\subseteq Z$ that still has positive measure. Even more, we can achieve that these sequences are uniformly bounded on the set $E$. Namely, because of the uniform boundedness and the uniform convergence of $a_N$ i $b_N$ on $E$ we have  \[\lim_{N\to\infty}\sup_{t\in E}|a_N(t)|<\infty,\enspace\lim_{N\to\infty}\sup_{t\in E}|b_N(t)|<\infty,\enspace\sup_{t\in E}|a(t)|<\infty,\enspace\sup_{t\in E}|b(t)|<\infty,\]   \[\lim_{N\to\infty}\sup_{t\in E}\left| a_N(t)-a(t)\right|=0,\enspace\lim_{N\to\infty}\sup_{t\in E}\left| b_N(t)-b(t)\right|=0.\]
Put $b(t):= \lim_{N\to\infty}b_N(t)$ and $a(t):= \lim_{N\to\infty}a_N(t)$, for $t\in E$. We know that \begin{align}\label{eq:limsupofbmnless}
\sup_{t\in E}\left| b_{M,N}(t)\right| &\le\sup_{t\in E}\left| a_M(t)-a(t)\right|\sup_{t\in E}\left| b_N(t)\right| +\sup_{t\in E}\left| a(t)\right| \sup_{t\in E}\left| b_N(t)-b(t) \right| \nonumber\\
&\quad +\sup_{t\in E}\left| b_M(t)\right| \sup_{t\in E}\left|a_N(t)-a(t)\right| +\sup_{t\in E}\left| a(t)\right| \sup_{t\in E}\left| b_M(t)-b(t) \right|,
\end{align} so by letting $M,N\to\infty$ in \eqref{eq:limsupofbmnless} we get \eqref{limsupb_MN}. 

As before, put $b_{M,N}(t)=\sum_{n\in F}D_ne^{2\pi int}$, where $F\subseteq \mathbb Z$ is the set of all frequencies for which the Fourier coefficients of $b_{M,N}$ are not equal to zero. We have  \begin{equation*}
\int_{E} \left|b_{M,N}(t) \right|^2\textnormal dt=|E|\sum_{n\in F}|D_n|^2+\sum_{\substack{n_1,n_2\in F\\n_1\neq n_2}} D_{n_1}\overline{D_{n_2}}\int_E e^{2\pi i(n_1-n_2)t}\textnormal dt .
\end{equation*}
For the first term in this sum we get \begin{equation*}
|E|\sum_{n\in F}|D_n|^2\ge |E|\sum_{j=M+1}^{N}|B_j|^2\prod_{\substack{M<k\le N\\ k\neq j}}A_k^2\ge |E|\sum_{j=M+1}^{N}|B_j|^2
\end{equation*} and, consequently, it follows that
\begin{equation}\label{integralb_mn}
|E|\sum_{j=M+1}^{N}|B_j|^2\le \int_{E} \left|b_{M,N}(t) \right|^2\textnormal dt+\Bigg|\sum_{\substack{n_1,n_2\in F\\n_1\neq n_2}} D_{n_1}\overline{D_{n_2}}\int_E e^{2\pi i(n_1-n_2)t}\textnormal dt\Bigg| .
\end{equation}
Now, for the second term in \eqref{integralb_mn}, by using the Cauchy--Schwarz inequality and Lemma \ref{lema_diff}, we have \begin{align}
&\Bigg| \sum_{\substack{n_1,n_2\in F\\n_1\neq n_2}} D_{n_1}\overline{D_{n_2}}\int_E e^{2\pi i(n_1-n_2)t}\textnormal dt\Bigg|=\Bigg|\sum_{\substack{n\in\mathbb Z\\|n|\ge m_{M+1}}} \Big( \sum_{\substack{n_1,n_2\in F\\ n_2-n_1=n}} D_{n_1}\overline{D_{n_2}}\Big) \int_E e^{-2\pi int}\textnormal dt\Bigg|   \nonumber \\ 
&\le \bigg(\sum_{n\in\mathbb Z}\Big|\sum_{\substack{n_1,n_2\in F\\ n_2-n_1=n}} D_{n_1}\overline{D_{n_2}}\Big|^2\bigg)^{\frac{1}{2}} \bigg(\sum_{\substack{n\in\mathbb Z\\ |n|\ge m_{M+1}}}\Big| \int_E e^{-2\pi int}\textnormal dt\Big|^2\bigg)^{\frac{1}{2}} \label{drugi_pribrojnik}.
\end{align}
Parseval's identity, applied to $1_E$, the characteristic function of the set $E$, implies $$\sum_{n\in\mathbb Z}\Big| \int_E e^{-2\pi int}\textnormal dt\Big|^2=\sum_{n\in\mathbb Z}\Big| \int_{\mathbb T} 1_E(t)e^{-2\pi int}\textnormal dt\Big|^2=\left\|1_E \right\|_{\textup L^2(\mathbb T)}^2=|E|<\infty.$$
So, in particular, for the second factor in \eqref{drugi_pribrojnik} we get \begin{equation}\label{lim_integrala}
\lim_{M\to\infty} \sum_{\substack{n\in\mathbb Z\\ |n|\ge m_{M+1}}}\Big| \int_E e^{-2\pi int}\textnormal dt\Big|^2 =0.
\end{equation}

Assume that $\sum_{j=1}^{\infty}|B_j|^2=\infty.$ Having in mind the equivalent conditions for having $\ell^p$ coefficients stated in the introduction, this actually means that we assumed the opposite of the desired conclusion, i.e.\ that our infinite product does not have $\ell^2$ coefficients. Furthermore, due to completeness, convergence at a single point $t\in E\subseteq Z$ implies that for that point \eqref{eq:limtoprove} holds. For $M=j-1$ and $N=j$ this implies $\lim_{j\to\infty}|B_j|=0$. As in the proof of Theorem \ref{Tmconvmetric}, we can easily find strictly increasing sequences of indices $(M_k)_{k\in\mathbb N}$ and $(N_k)_{k\in\mathbb N}$ such that \begin{equation} \label{eq:lastintheproof}
\lim_{k\rightarrow\infty}\sum_{j=M_k+1}^{N_k}|B_j|^2=1.\end{equation}
Inequalities \eqref{integralb_mn} and \eqref{drugi_pribrojnik} and Lemma \ref{lema_ocjena_Dn-ova} imply 
\begin{equation*}
|E|\sum_{j=M_k+1}^{N_k}|B_j|^2\le |E|\left( \sup_{t\in E }\left|b_{M_k,N_k}(t)\right|\right)^2 +e^{4\sum_{j=M_k+1}^{N_k}|B_j|^2}\bigg(\sum_{\substack{n\in\mathbb Z\\ |n|\ge m_{M_k+1}}}\Big| \int_E e^{-2\pi int}\textnormal dt\Big|^2\bigg)^{\frac{1}{2}}.
\end{equation*}
Letting $k\to\infty$ in the above inequality, from \eqref{limsupb_MN}, \eqref{lim_integrala}, and \eqref{eq:lastintheproof} we get $|E|\le 0$, which is a contradiction. This proves Theorem \ref{Tmconvae2}.
\end{proof}

\section*{Acknowledgments}
This work was supported in part by the Croatian Science Foundation under the project UIP-2017-05-4129 (MUNHANAP). The author would like to thank her supervisor Vjekoslav Kova\v{c} for his help and useful suggestions.


\bibliographystyle{alpha}

\end{document}